\documentclass[11pt]{amsart}

\usepackage[T1]{fontenc}
\usepackage[american]{babel}

\usepackage[colorlinks=true, pdfstartview=FitV, linkcolor=blue, 
      citecolor=blue]{hyperref}

\usepackage{bbm, a4wide}
\usepackage{amssymb}
\usepackage{csquotes}
\usepackage{booktabs} 
\usepackage{tabularx}
\usepackage{amsmath}
\usepackage{mathtools}
\usepackage{amsthm}
\usepackage{graphicx}
\graphicspath{ {./graphics/} }
\usepackage{hyperref}
\usepackage{enumitem}

\usepackage{cleveref}
\newcommand{\eqrefn}[1]{(\ref{#1})}

\theoremstyle{plain}
\newtheorem{theorem}{Theorem}[section]

\newtheorem{lemma}[theorem]{Lemma}
\newtheorem{proposition}[theorem]{Proposition}
\newtheorem{corollary}[theorem]{Corollary}

\theoremstyle{definition}

\newtheorem{remark}[theorem]{Remark}

\renewcommand{\epsilon}{\varepsilon}
\renewcommand{\leq}{\leqslant}
\renewcommand{\geq}{\geqslant}
\renewcommand{\setminus}{\backslash}
\renewcommand{\Re}{{\rm Re}}

\renewcommand{\P}{\mathbb{P}}

\renewcommand{\d}{~\mathsf{d}}

\newcommand{\N}{\mathbb{N}}
\newcommand{\Z}{\mathbb{Z}}
\newcommand{\R}{\mathbb{R}}
\newcommand{\C}{\mathbb{C}}

\title[Truncated pretentious distances of multiplicative functions]
{Truncated pretentious distances of \\[1mm] multiplicative arithmetic functions}

\author{Thomas Renard}
\address{%
\begin{tabular}{@{}l@{}}
Fakultät für Mathematik, Universität Bielefeld,\\
Postfach 100131, 33501 Bielefeld, Germany
\end{tabular}%
}
\email{thomas.m.renard@gmail.com}

\begin{document}
\begin{abstract}
Let $f\colon\mathbb{N}\to\mathbb{U}$ be a non-pretentious multiplicative function in the sense of Granville and Soundararajan. We study the convergence of the autocorrelation averages of $f$. In particular, we investigate the relationship between the local pretentious behaviour of $f$ on intervals of the form $[x^{\varepsilon},x]$ and its deviation from $1$ on the primes. As an application, we exhibit a class of non-pretentious multiplicative functions for which Elliott's conjecture in its original formulation holds, following up on a result by Klurman, Mangerel and Teräväinen.
\end{abstract}
\maketitle
\section{Introduction}
\label{introduction}
If $f \colon\N \to \mathbb{U}$ is a multiplicative function, where $\mathbb{U}$ denotes the closed unit disk of the complex plane, we are interested in determining whether the correlation average
\begin{equation}\label{eq_correlation_averages}
    \frac{1}{x}\sum_{n\leq x}f(a_1n+h_1)\cdots f(a_kn+h_k),
\end{equation}
converges to $0$ or not as $x\to \infty$, where $a_1,\dots, a_k$ and $h_1,\dots,h_k$ are some fixed integers.

It has long been known that the convergence of the correlation average \eqrefn{eq_correlation_averages} is based on whether the function $f$ is pretentious or not. Given two arithmetic functions $g_1, g_2 \colon \N \to \mathbb{U}$, we define the pretentious distance of Granville and Soundararajan \cite{granvillebook} by:
    \begin{align*}
        \mathbb{D}(g_1,g_2; x) := \left( {\displaystyle \sum_{p\leq x} \frac{1-\Re(g_1(p)\overline{g_2(p)})}{p}} \right)^{1/2},
    \end{align*}
    for some real number $x> 0$.
     We can similarly define the truncated pretentious distance $\mathbb{D}(g_1,g_2;x,y)$ for real numbers $0<x<y$ by taking the sum over the primes $x<p\leq y$.  
     We say that $f$ is \emph{pretentious} if there exists a Dirichlet character $\chi$ and $t\in \R$ such that 
    \begin{align*}
        \mathbb{D}(f, n \mapsto \chi(n)n^{it}; \infty) < \infty.
    \end{align*}  
    If $f$ is not pretentious, then we call it \emph{non-pretentious}.
Whenever we have $\mathbb{D}(f,g; \infty )< \infty$, where $g \colon \N \to \mathbb{U}$ is a multiplicative function, we say that $f$ \emph{pretends to be} $g$. 

One of the ways in which the pretentious distance is naturally involved in the study of multiplicative functions can be seen with Elliott's conjecture \cite{Elliott1994OnTC}. The latter states that, given multiplicative functions $f_1,\dots,f_k\colon\mathbb N\to\mathbb U$ and fixed integers $a_1,\dots,a_k,h_1,\dots,h_k$, with $a_i h_j\neq a_j h_i$ for all $i\neq j$,
if at least one of the functions $f_1,\dots,f_k$ is non-pretentious, then the correlation average
\begin{align}\label{eq_elliott_conjecture}
\frac{1}{x}\sum_{n\leq x}
f_1(a_1n+h_1)\cdots f_k(a_kn+h_k)
\end{align}
converges to $0$ as $x\to\infty$.
    Throughout this paper, we say that a multiplicative function $f\colon\N\to \mathbb{U}$ satisfies Elliott's conjecture if, for any $k\in \N$, the autocorrelation averages \eqrefn{eq_correlation_averages} converge to $0$ as $x\to \infty$.

The original formulation of this conjecture is false on a small technicality, the construction of a counterexample being due to Matomäki, Radziwiłł, and Tao \cite{Matom_ki_2015}. However, it has been proved for some specific cases. The one-point case ($k=1$) is known to be essentially Halász's theorem \cite{Halsz1968berDM}. In a recent breakthrough from 2016, Tao established \cite{tao2016logarithmicallyaveragedchowlaelliott} the two-point case of the logarithmically averaged Elliott conjecture. Tao and Teräväinen subsequently obtained a more general structural theorem \cite{Tao_Tera_2019} for logarithmically averaged correlations, giving a proof of a weaker form of the logarithmically averaged Elliott conjecture. The same authors also proved \cite{Tao_2019} some cases of the unweighted Elliott conjecture at almost all scales.

In a recent development, Klurman, Mangerel and Teräväinen \cite{klurman2023elliottsconjectureapplications} established a quantitative bound for the correlation averages \eqrefn{eq_elliott_conjecture} of multiplicative functions in terms of the truncated pretentious distances of the form $\mathbb{D}(f_j,\chi_j(n)n^{it_j};x^{\epsilon},x)$, providing new progress towards Elliott's conjecture. 
In particular, if $f\colon \N \to \mathbb{U}$ is a non-pretentious multiplicative function such that $\mathbb{D}(f,\chi(n)n^{it};x^{\epsilon},x) = o_{x\to \infty}(1)$ for some twisted Dirichlet character $\chi(n)n^{it}$, their upper bound implies that $f$ satisfies Elliott's conjecture. 

As a consequence of this result, the same authors answered a question asked by de la Rue at the AIM conference on Sarnak's conjecture \cite{IASsarnak}, namely whether it is possible to find a sufficiently large subset of primes for which the autocorrelation average of the associated Liouville-like function vanishes. For any subset of primes $\mathcal{P}\subset \P$, the functions $\Omega_{\mathcal{P}}$ and $\omega_{\mathcal{P}}$ denote the number of prime divisors with or without multiplicities, respectively.
\begin{theorem}[{\cite[Thm~2.1]{klurman2023elliottsconjectureapplications}}]\label{thm_klurman_mangerel_tera_liouville_like}
    Let $\lambda_{\mathcal{P}}(n) = (-1)^{\Omega_{\mathcal{P}}(n)}$ or $\lambda_{\mathcal{P}}(n) = (-1)^{\omega_{\mathcal{P}}(n)}$, where $\mathcal{P}\subset \P$ has relative density $0$ in $\P$ and $\sum_{p\in \mathcal{P}} \frac{1}{p}=\infty$. Then, $\lambda_{\mathcal{P}}$ satisfies Elliott's conjecture.
\end{theorem}

In this paper, we explore some consequences of this result, whose proofs use similar ideas and tools.
Observing that we can express a Liouville-like function as $ \lambda_{\mathcal{P}}(n) = e(\Omega_{\mathcal{P}}(n)\alpha)$, where $\alpha = 1/2$, this naturally leads us to consider a broader class of multiplicative functions obtained by assigning rotations to the prime factors belonging to sparse sets of primes. In particular, we prove the following result.
\begin{theorem}\label{thm_elliott_for_functions_close_to_one}
    Let $f\colon \N \to \mathbb{U}$ be a non-pretentious multiplicative function such that the set $\mathcal{P}=\{p\in \mathbb{P}: f(p)\neq 1\}$ has relative density $0$ in the primes. Then, $f$ satisfies Elliott's conjecture.
\end{theorem}

 We also investigate some limitations of the underlying approach based on the size of the truncated pretentious distance. Considering Chowla's conjecture \cite{Chowla1966TheRH}, it is natural to ask whether the result can be extended to larger subsets of the primes. Unfortunately, in the finitely generated setting (in the sense that the set $\{f(p) : p\in \P\}$ is finite), there is an obstruction. Any multiplicative function that is arbitrarily close to $1$ on $[x^{\epsilon},x]$, while remaining globally non-pretentious, can only differ from $1$ on a set of primes of lower relative density $0$.

\begin{theorem}\label{thm_cool_criteria}
    Let $f \colon \N \to \mathbb{U} $ be a non-pretentious finitely generated multiplicative function. Let $\mathcal{P} = \{p\in \mathbb{P}:f(p) \neq 1\}$. Then, if there exists an $\epsilon := \epsilon(x)$ taking values in $(1/\log\log x, 1/2)$ such that $\mathbb{D}(f,1;x^{\epsilon},x)=o_{x\to \infty}(1)$, the set $\mathcal{P}$ has lower relative density $0$ in the primes.   
\end{theorem}

\section{Notations and some preliminaries}

If not specified otherwise, $p$ denotes a prime number, and $N$, $n$, $m$ and $k$ are integers. The expression $n|m$ means that $n$ divides $m$, $p^k \| n$ means $p^k | n$ but $p^{k+1}$ does not divide $n$, and $(n,m) = k$ means that the greatest common divisor of $n$ and $m$ is $k$. 

For any set $A\subset \N$ and $x>0$, we use the notation $\pi_A(x)$ to denote the number of primes less than or equal to $x$ that are in $A$, i.e $\pi_A(x) = |\{p\in [1,x] \cap A\}|$. We simply write $\pi(x)$ for $\pi_{\mathbb{P}}(x)$. We denote by $\pi(x;n,m) = |\{p\in [1,x] \cap \P:p\equiv m\mod n\}|$, i.e, the number of primes $p\leq x$ congruent to $m$ modulo $n$. We denote by $\varphi$ Euler's totient function. For all $x\in \R$ we write $e(x) := \exp(2\pi i x)$. We use the standard asymptotic notations. Additionally, we write $f(x) \sim g(x)$ if $f(x) = g(x)+o(1)$ and $f(x) \asymp g(x)$ if $f \ll g$ and $g\ll f$ hold simultaneously.

Throughout this paper, we use the following notions of density. Let $\mathcal{P}\subset \mathbb{P}$ be a subset of primes. The natural relative density of $\mathcal{P}$ in the primes is defined as
\begin{align*}
    d_{\mathbb{P}}(\mathcal{P}) = \lim_{N\to \infty} \frac{|\mathcal{P}\cap \{1,\dots,N\}|}{|\mathbb{P}\cap \{1,\dots,N\}|},
\end{align*}
if the limit exists. We also define the upper (resp. lower) relative density of $\mathcal{P}$ by replacing the $\lim$ by the $\limsup$ (resp. $\liminf$), denoted $\overline{d_{\mathbb{P}}}$ (resp. $\underline{d_{\mathbb{P}}}$). We define its logarithmic relative density to be
\begin{align*}
    d_{\mathbb{P}}^{\log}(\mathcal{P}) &= \lim_{x\to \infty}\frac{1}{\log\log x}\sum_{\substack{p\leq x \\p\in \mathcal{P}}}\frac{1}{p},
\end{align*}
if the limit exists. Similarly, we define the upper (resp. lower) logarithmic relative density of $\mathcal{P}$ by replacing the $\lim$ by the $\limsup$ (resp. $\liminf$), denoted $\overline{d^{\log}_{\mathbb{P}}}$ (resp. $\underline{d^{\log}_{\mathbb{P}}}$).

We note the following well-known relationship between these densities (see for example \cite[Thm.~III.1.2]{tenenbaum1995introduction} for a proof of these inequalities). For any subset of primes $\mathcal{P}\subset \mathbb{P}$, we have
    \begin{equation}\label{eq_densities}
        \underline{d_{\mathbb{P}}}(\mathcal{P}) \leq \underline{d^{\log}_{\mathbb{P}}}(\mathcal{P}) \leq \overline{d^{\log}_{\mathbb{P}}}(\mathcal{P}) \leq \overline{d_{\mathbb{P}}}(\mathcal{P}).
    \end{equation}
    In particular, if $\mathcal{P}$ has natural relative density, then it also has logarithmic relative density, and the two are equal.

Finally, since $\mathbb{D}(f,f;x)$ can be non-zero and $\mathbb{D}(f,g;x)$ can vanish without $f$ and $g$ being equal, the pretentious distance is not a metric. However, it still behaves close enough to a metric. In particular, it satisfies the triangle inequality (see e.g. \cite[Lemma~3.1]{granvillebook}). Let $f,g,h \colon \N \to \mathbb{U}$ be multiplicative functions. For any $x>0$, we have
    \begin{equation*}
        \mathbb{D}(f,g;x) \leq \mathbb{D}(f,h;x) + \mathbb{D}(h,g;x).
    \end{equation*}

\section{Classes of multiplicative functions satisfying Elliott's conjecture}\label{sec_classes_of_multiplicative_functions_satisfying_elliott}
We begin by introducing the example mentioned in the introduction, which generalises the class of Liouville-like functions. We show that these functions satisfy Elliott's conjecture. Our methods use arguments similar to the ones used in \cite{klurman2023elliottsconjectureapplications} to prove Theorem \ref{thm_klurman_mangerel_tera_liouville_like}.

\begin{proposition}\label{prop_example_rotated_liouville}
    Let $\ell \in \N$, $\alpha_1, \dots, \alpha_{\ell} \in \R \setminus \Z$, and $\mathcal{P}_1, \dots, \mathcal{P}_{\ell}\subseteq \mathbb{P}$ be subsets of prime with 
    \begin{equation*}
        \sum_{\substack{p\in \P \\ \alpha(p) \notin \Z}}  \frac{1}{p}=\infty ,\quad\text{where } \alpha(p) = \sum_{\substack{1\leq i \leq \ell \\ p\in \mathcal{P}_i}} \alpha_i,
    \end{equation*}
    and $d_{\mathbb{P}}(\mathcal{P}_i) = 0$ for $i=1,\dots,\ell$. Let $u\colon \N \to \R$ be defined by $u(n) = \sum_{i=1}^{\ell} u_{\mathcal{P}_i}(n)\alpha_i$, where $u_{\mathcal{P}_i}$ is either $\omega_{\mathcal{P}_i}$ or $\Omega_{\mathcal{P}_i}$ for $i=1,\dots,\ell$, and let $f\colon \N \to \mathbb{U}$ be defined by $f(n) = e(u(n))$. Then, $f$ satisfies Elliott's conjecture.
\end{proposition}

We use the following two lemmas.

\begin{lemma}\label{lemma_technical_result}
    Let $\mathcal{P} \subset \mathbb{P}$ be a subset of primes of relative density $0$. Then, for any fixed $\epsilon >0$ 
    \begin{equation*}
        \sum_{\substack{x^{\epsilon }\leq p \leq x\\p\in \mathcal{P}}}\frac{1}{p} = o(1).
    \end{equation*}
\end{lemma}
\begin{proof}
    Let $J = (1-\epsilon)\log_2 (x)$. Without loss of generality, we can assume $J\in \mathbb{N}$ (otherwise we round using $\lceil J\rceil$ since this rounding will have a negligible impact). We write 
    \begin{equation*}
        (x^{\epsilon},x] = \bigcup_{j=0}^{J-1}\left(\frac{x}{2^{j+1}},\frac{x}{2^j}\right],
    \end{equation*}
    so that we have
    \begin{equation*}
        \sum_{\substack{x^{\epsilon }\leq p \leq x\\p\in \mathcal{P}}}\frac{1}{p} = \sum_{j=0}^{J-1}\sum_{\substack{x/2^{j+1}\leq p \leq x/2^j\\p\in \mathcal{P}}}\frac{1}{p}.
    \end{equation*}
    Let $I_j = \left(\frac{x}{2^{j+1}},\frac{x}{2^j}\right]$ for $j=0,\dots,J-1$. Then, since for any primes $p\in I_j$ we have $\frac{1}{p}\leq \frac{2^{j+1}}{x}$, it follows that
    \begin{equation*}
        \sum_{\substack{p\in I_j\\p\in \mathcal{P}}}\frac{1}{p}\leq \frac{2^{j+1}}{x}|\{p\in I_j : p\in \mathcal{P}\}|.
    \end{equation*}
    Now, notice that $|\{p\in I_j : p\in \mathcal{P}\}| = \pi_{\mathcal{P}}(x/2^j)-\pi_{\mathcal{P}}(x/2^{j+1})\sim \pi_{\mathcal{P}}(x/2^j)/2 =o(\pi(x/2^j))$. It follows that
    \begin{equation*}
        \sum_{\substack{p\in I_j\\p\in \mathcal{P}}}\frac{1}{p} = o\left(\frac{1}{\log(x/2^j)}\right),
    \end{equation*}
    and by summing over $j$, we conclude that
    \begin{equation*}
        \sum_{\substack{x^{\epsilon }\leq p \leq x\\p\in \mathcal{P}}}\frac{1}{p} = o(1).
    \end{equation*}
    In fact, using summation by parts together with $J= (1-\epsilon)\log_2 (x) $, we get
    \begin{align*}
        \sum_{j=0}^J\frac{1}{\log(x/2^j)} &= \frac{J}{\log(x/2^J)}-\log 2 \int_0^J\frac{t}{(\log(x/2^t))^2}\d t +O\left(\frac{1}{\log x}\right)\\
        &= \int_0^J\frac{\d t}{\log(x/2^t)}+O\left(\frac{1}{\log x}\right)= \frac{\log(1/\epsilon)}{\log 2}+O\left(\frac{1}{\log x}\right),      
    \end{align*}
    where we conclude by noticing that the main term is a constant only depending on $\epsilon$.
\end{proof}

\begin{lemma}\label{lemma_sum_pit}
    Let\/ $t\neq 0$ be a real number and let $q\in \N$ be fixed. Then, for all integers $0< \ell <q$ such that $(\ell,q) = 1$, we have that
    \begin{equation*}
        \sum_{\substack{p\leq x\\p\equiv \ell \mod q}}\frac{1}{p^{1+it}}
    \end{equation*}
    converges to a constant depending on $q,\ell$ and $t$ as $x\to \infty$.
\end{lemma}
\begin{proof}
    Using summation by parts, we get for all $x\geq 2$
    \begin{equation}\label{eq_1}
        \sum_{\substack{p\leq x\\p\equiv \ell \mod q}}\frac{1}{p^{1+it}} = \frac{\pi(x;q,\ell)}{x^{1+it}}+(1+it)\int_2^x\frac{\pi(u;q,\ell)}{u^{2+it}}\d u.
    \end{equation}
    From the \emph{Prime Number Theorem} (PNT) for arithmetic progressions, we get the estimate
    \begin{equation*}
        \pi(x;q,\ell) = \frac{x}{\varphi(q)\log x}+ O_q\left(\frac{x}{(\log x)^2}\right),
    \end{equation*}
    which we can insert in \eqrefn{eq_1} to get 
    \begin{equation*}
        \sum_{\substack{p\leq x\\p\equiv \ell \mod q}}\frac{1}{p^{1+it}} = \frac{\pi(x;q,\ell)}{x^{1+it}}+\frac{(1+it)}{\varphi(q)}\int_2^x\frac{u^{-1-it}}{\log u}\d u + O_{q,t}\left(\int_2^x \frac{\d u}{u(\log u)^2}\right).
    \end{equation*}
    Clearly, we have
    \begin{equation*}
        \frac{\pi(x;q,\ell)}{x^{1+it}}= o_{x\to \infty}(1).
    \end{equation*}
    Moreover, the integral from the error term converges absolutely. As for the main integral, using a suitable change of variable and integration by parts, we get
    \begin{equation*}
       \int_2^x\frac{u^{-1-it}}{\log u}\d u =C_t -\frac{1}{it}\int_{\log 2}^{\log x} \frac{e^{-itv}}{v^2}\d v +O_t\left(\frac{1}{\log x}\right),
    \end{equation*}
    where $C_t\in \C$ is a fixed constant depending only on $t$, and the remaining integral converges absolutely as $x\to \infty$, which completes the argument.
\end{proof}
We use these lemmas together with the following result from Klurman, Mangerel and Teräväinen \cite{klurman2023elliottsconjectureapplications} to prove Proposition \ref{prop_example_rotated_liouville}. We note that this result also plays a key role in several other arguments throughout the paper.
\begin{theorem}[{\cite[Thm~4.1]{klurman2023elliottsconjectureapplications}}]\label{thm_klurman_mangerel_tera}
    Let $k\geq 1$ and $a_1,\dots, a_k,h_1,\dots,h_k\in \N$ be fixed with $a_ih_j\neq a_jh_i$ whenever $i\neq j$. Also, let Dirichlet characters $\chi_1,\dots,\chi_k$ and real numbers $t_1,\dots,t_k$ be fixed. Then, for any $x\geq 3$, $\epsilon \in (1/\log\log x, 1/2)$ and any multiplicative functions $f_1,\dots f_k \colon \N \to \mathbb{U}$ satisfying
    \begin{equation*}
    \max_{1\leq j\leq k} \mathbb{D}(f_j,\chi_j(n)n^{it_j};x^{\epsilon},x) \leq \epsilon,
    \qquad \text{and}\qquad
    \max_{1\leq j\leq k} \mathbb{D}(f_j,\chi_j(n)n^{it_j};x) \geq \frac{1}{\epsilon},
\end{equation*}
    we have
    \begin{equation}\label{eq_KMT}
        \left| \frac{1}{x}\sum_{n\leq x} f_1(a_1n+h_1)\cdots f_k(a_kn+h_k)\right| \ll \left(\log \frac{1}{\epsilon}\right)^{1/2}\epsilon.
    \end{equation}
\end{theorem}
\begin{proof}[Proof of Proposition $\ref{prop_example_rotated_liouville}$]
    To show that $f$ is non-pretentious, we observe 
    \begin{equation*}
        \mathbb{D}(f,1;\infty)^2 = \sum_{p\in \mathcal{P}_1\cup \cdots \cup \mathcal{P}_\ell}\frac{1-\Re(f(p))}{p}. 
    \end{equation*}
   Setting 
    \begin{equation*}
        m = \min_{ \substack{J\subseteq \{1,\dots,\ell\}\\ \sum_{j\in J} \alpha_j \notin \Z}} \biggl( 1-\Re\Big(e\Big(\sum_{j\in J}\alpha_j\Big)\Big) \biggl) > 0,
    \end{equation*}
    and using the assumption on the sum of reciprocal of primes, we have
    \begin{equation*}
        \mathbb{D}(f,1;\infty)^2 
        \geq m\sum_{\substack{p\in \P\\ \alpha(p) \notin \Z}}\frac{1}{p}= \infty.
    \end{equation*}
    Using the fact that any Dirichlet character $\chi$ can only take finitely many values on the primes, we can show using a similar argument that $\mathbb{D}(f,\chi;\infty) = \infty$. 
    
    Let $\mathcal{P}=\bigcup_{i=1}^{\ell}\mathcal{P}_i$, which has relative density $0$ in the primes from the assumption on $\mathcal{P}_i$. For $t\neq 0$, we see that
    \begin{align*}
        \mathbb{D}(f,\chi(n)n^{it};x)^2 &\geq \sum_{\substack{p\leq x \\ p\notin \mathcal{P}}}\frac{1-\Re\left(\overline{\chi(p)}p^{-it}\right)}{p} = \mathbb{D}(1,\chi(n)n^{it};x)^2 - \sum_{\substack{p\leq x \\ p\in \mathcal{P}}}\frac{1-\Re\left(\overline{\chi(p)}p^{-it}\right)}{p},
    \end{align*}
    where the last term is bounded as follows
    \begin{equation*}
        \sum_{\substack{p\leq x \\ p\in \mathcal{P}}}\frac{1-\Re\left(\overline{\chi(p)}p^{-it}\right)}{p} \leq 2\sum_{\substack{p\leq x \\ p \in \mathcal{P}}}\frac{1}{p}.
    \end{equation*}
    Using summation by parts together with the PNT, we see that
    \begin{align*}
        \sum_{\substack{p\leq x \\ p \in \mathcal{P}}}\frac{1}{p} = o(\log\log x).
    \end{align*}
    Finally, assuming that $\chi$ is $q$-periodic, we have
    \begin{equation*}
        \mathbb{D}(1,\chi(n)n^{it};x)^2 \geq \sum_{\substack{0<\ell < q\\ (\ell,q) =1}}\sum_{\substack{p\leq x\\p\equiv \ell \mod q}}\frac{1}{p}-\Re\left(\sum_{\substack{0< \ell< q\\ (\ell,q) =1}}\overline{\chi(\ell)}\sum_{\substack{p\leq x\\p\equiv \ell \mod q}}\frac{1}{p^{1+it}}\right).
    \end{equation*}
    Using the PNT for arithmetic progressions together with Lemma \ref{lemma_sum_pit}, we get
    \begin{equation*}
         \mathbb{D}(1,\chi(n)n^{it};x)^2 \geq \log\log x +O_{q,\ell,t}(1).
    \end{equation*}
    Putting everything together and letting $x\to \infty$ gives $\mathbb{D}(f,\chi(n)n^{it};\infty) = \infty$. Consequently, $f$ is non-pretentious.

    Now, observe that, for any $\epsilon >0$, we have
    \begin{align*}
        \mathbb{D}(f,1;x^{\epsilon},x)^2 &= \sum_{x^{\epsilon}<p\leq x} \frac{1-\Re(f(p))}{p}\leq \sum_{i=1}^{\ell}\sum_{\substack{x^{\epsilon}<p\leq x\\p\in \mathcal{P}_i}} \frac{2}{p} = o_{x\to \infty}(1),
    \end{align*}
     where we used Lemma \ref{lemma_technical_result} together with the density assumption on the sets $\mathcal{P}_1,\dots, \mathcal{P}_{\ell}$ for the last step. Therefore, there exists some function $\epsilon:= \epsilon(x)$ with values in $ (1/\log\log x, 1/2)$ converging to zero slowly enough as $x\to \infty$  so that
    \begin{equation*}
        \mathbb{D}(f,1;x^{\epsilon(x)},x)^2 = o(1).
    \end{equation*}
    An application of Theorem \ref{thm_klurman_mangerel_tera} finishes the proof.   
\end{proof}
\begin{remark}
    If the $\mathcal{P}_i$ are pairwise disjoint, $f$ does not pretend to be any of the individual factors $e(u_{\mathcal{P}_i}(n)\alpha_i)$. In fact, suppose we compare $f$ with one of its factors $g_j(n) = e(u_{\mathcal{P}_j}(n)\alpha_j)$, with $j\in \{1,\dots,\ell \}$. Then on the primes we have
    \begin{equation*}
        g_j(p) =
        \begin{cases}
            e(\alpha_j), & p\in \mathcal{P}_j,\\
            1, & p\notin \mathcal{P}_j,
        \end{cases}
    \end{equation*}
    from which we deduce that
    \begin{equation*}
        f(p)\overline{g_j(p)} = 
        \begin{cases}
            1, & p\in \mathcal{P}_j,\\
            e(\alpha_i), & p\in \mathcal{P}_i, i\neq j,\\
            1, & p\notin \bigcup_{i}\mathcal{P}_i.
        \end{cases}
    \end{equation*}
    Therefore, we have 
    \begin{equation*}
        \mathbb{D}(f,g_j;x)^2 = \sum_{i\neq j}\left(1-\cos(2\pi \alpha_i)\right)\sum_{\substack{p\leq x\\p\in \mathcal{P}_i}}\frac{1}{p},
    \end{equation*}
    which we know to diverge as $x\to \infty$ from the assumptions on the sets $\mathcal{P}_i$. If we do not assume the $\mathcal{P}_i$ to be pairwise disjoint, it is possible for $f$ to pretend to be one of its factors. Indeed, let $\mathcal{P}\subset \P$ be a subset of primes of relative natural density $0$ in the primes such that $\sum_{p\in \mathcal{P}} 1/p = \infty$. Let $\mathcal{P}_1 = \cdots = \mathcal{P}_{\ell} = \mathcal{P}$, with $\ell \geq 3$, and for any $i=1,\dots,\ell$, let $\alpha_i = 1/(\ell-1)\in \R\setminus \Z$. We have $\sum_{p\in \P, \alpha(p)\notin \Z} 1/p = \infty$, but for any primes $p$ and any $j\in \{1,\dots,\ell\}$, we have $f(p)\overline{g_j(p)}=1$, which implies $\mathbb{D}(f,g_j;x) = 0$ for any $x>0$. Hence, $f$ pretends to be each one of its factors $g_j$. In order for this observation to remain true in the case where the $\mathcal{P}_i$ are not pairwise disjoint we would need a stronger condition on the sum of reciprocal of primes $\sum_{p\in \P, \alpha(p)\notin \Z}1/p$, such as 
    \begin{equation*}
        \sum_{\substack{p\in \P\\ \alpha(p)-\alpha_j1_{\mathcal{P}_j}(p)\notin \Z}}\frac{1}{p} = \infty,
    \end{equation*}
    for all $j=1,\dots,\ell$.
\end{remark}

Theorem \ref{thm_elliott_for_functions_close_to_one} is proved in a similar way.

\begin{proof}[Proof of Theorem $\ref{thm_elliott_for_functions_close_to_one}$]
    Using Lemma \ref{lemma_technical_result}, we see that, for any $\epsilon>0$,
    \begin{equation*}
        \mathbb{D}(f,1;x^{\epsilon},x)^2 \leq \sum_{\substack{x^{\epsilon}<p\leq x\\p\in \mathcal{P}}}\frac{2}{p} = o_{x\to\infty}(1).
    \end{equation*}
    Since $f$ is non-pretentious, the conclusion immediately follows by using Theorem \ref{thm_klurman_mangerel_tera}, as we saw in the proof of Proposition \ref{prop_example_rotated_liouville}.
\end{proof}

The intuition behind the pretentious distance is that two functions $f$ and $g$ that pretend to be each other should exhibit similar statistical behaviour. In particular, if $g$ is a function that is close enough to $1$, it satisfies Elliott's conjecture, and so do all the non-pretentious multiplicative functions pretending to be $g$. 
\begin{corollary}\label{corollary_pretend_elliott}
    Let $f\colon \N \to \mathbb{U}$ be a non-pretentious multiplicative function. Assume there is a non-pretentious multiplicative function $g\colon \N \to \mathbb{U}$ such that $f$ pretends to be $g$ and $d_{\mathbb{P}}(\mathcal{P}) = 0$, where $\mathcal{P} = \{p\in \mathbb{P}: g(p) \neq 1\}$.
    Then, $f$ satisfies Elliott's conjecture.
\end{corollary}
\begin{proof}
    For any $\epsilon >0$, the triangular inequality implies
    \begin{equation*}
        \mathbb{D}(f,1;x^{\epsilon},x) \leq  \mathbb{D}(f,g;x^{\epsilon},x) + \mathbb{D}(g,1;x^{\epsilon},x).
    \end{equation*}
    We know from the proof of Theorem \ref{thm_elliott_for_functions_close_to_one} that $\mathbb{D}(g,1;x^{\epsilon},x) = o_{x\to \infty}(1)$. Moreover, since $f$ pretends to be $g$, we also have $\mathbb{D}(f,g;x^{\epsilon},x) = o_{x\to \infty}(1)$, as this is the tail of a converging series. Therefore, we get $\mathbb{D}(f,1;x^{\epsilon},x) = o_{x\to \infty}(1)$ and conclude by using Theorem \ref{thm_klurman_mangerel_tera}.
\end{proof}

\section{A sparsity criterion for some non-pretentious functions}\label{sec_a_sparsity_criterion}

In this section, we investigate the limitations of Theorem \ref{thm_klurman_mangerel_tera} in proving that certain functions have vanishing autocorrelations. In particular, we prove Theorem \ref{thm_cool_criteria} where we show that non-pretentious, finitely generated, multiplicative functions $f$ with $\mathbb{D}(f,1;x^{\epsilon},x)=o_{x\to \infty}(1)$ must take values different from $1$ on a subset of primes of $0$ lower relative density.

\begin{proof}[Proof of Theorem $\ref{thm_cool_criteria}$]
    Assume to the contrary that $\underline{d_{\mathbb{P}}}(\mathcal{P}) > 0$. Since $f$ is finitely generated, we can decompose $\mathcal{P} = \bigcup_{i=1}^m \mathcal{P}_i$, where $\mathcal{P}_i =\{p\in \mathbb{P}: f(p) = h_i\}$ for $i=1,\dots,m$, and $\{h_1,\dots,h_m\}\subseteq \mathbb{U} \setminus \{1\}$ is the set of all the possible values for $f(p)$ that differ to $1$. Since $\underline{d_{\mathbb{P}}}(\mathcal{P})>0$, at least one of the $\mathcal{P}_i$'s is such that $\underline{d_{\mathbb{P}}}(\mathcal{P}_i)>0$, so assume without loss of generality that $\underline{d_{\mathbb{P}}}(\mathcal{P}_1)=c>0$.
    Now let $\epsilon := \epsilon(x) $ be a function of $x$ taking values in $(1/\log\log x, 1/2)$. We have
    \begin{equation*}
        \mathbb{D}(f,1;x^{\epsilon},x)^2 \geq (1-\Re(h_1))\sum_{\substack{x^{\epsilon} \leq p \leq x\\p \in \mathcal{P}_1}}\frac{1}{p}.
    \end{equation*}
    Similarly as in the proof of Lemma \ref{lemma_technical_result}, we let $J = (1-\epsilon)\log_2 (x)$ and decompose
    \begin{equation*}
        (x^{\epsilon},x] = \bigcup_{j=0}^{J-1}\left(\frac{x}{2^{j+1}}, \frac{x}{2^j}\right],
    \end{equation*}
    so that
    \begin{equation*}
        (1-\Re(h_1))\sum_{\substack{x^{\epsilon} \leq p \leq x\\p \in \mathcal{P}_1}}\frac{1}{p} =(1-\Re(h_1)) \sum_{j=0}^{J-1}\sum_{\substack{p\in I_j\\p\in \mathcal{P}_1}}\frac{1}{p},
    \end{equation*}
    where $I_j = \left(\frac{x}{2^{j+1}}, \frac{x}{2^j}\right]$ for $j=0,\dots, J-1$. Using now the lower density assumption together with the PNT, we see that, for $x$ large enough, we have
    \begin{align*}
        \sum_{\substack{p\in I_j\\p\in \mathcal{P}_1}}\frac{1}{p} &\geq \frac{2^j}{x}|\{p\in I_j: p\in \mathcal{P}_1\}|\gg_c\frac{2^j}{x}\cdot \pi(x/2^j) \sim \frac{1}{\log x/2^j}.
    \end{align*}
        We then have
        \begin{equation*}
            \mathbb{D}(f,1;x^{\epsilon},x)^2 \gg_c (1-\Re(h_1))\log 1/\epsilon,
        \end{equation*}
        as $x\to \infty$, and the conclusion follows.
    \end{proof}

    If the density $d_{\mathbb{P}}(\mathcal{P})$ of the subset $\mathcal{P}\subset \mathbb{P}$ as defined in Theorem~\ref{thm_cool_criteria} exists, the previous result gives the equivalence between $d_{\mathbb{P}}(\mathcal{P})=0$ and $\mathbb{D}(f,1;x^{\epsilon},x)=o(1)$, for some $ \epsilon:= \epsilon(x)$ taking values in $ (1/\log\log x, 1/2)$ converging to zero slowly enough.

If there exist functions $f$ with $d_{\mathbb P}(\mathcal P')\neq0$, where $\mathcal P'=\{p\in\mathbb P:f(p)\neq1\}$, that pretend to be sufficiently close to $1$, then Corollary \ref{corollary_pretend_elliott} would extend the class of functions known to satisfy Elliott's conjecture. This raises the question of whether such functions exist. As a consequence of Theorem \ref{thm_cool_criteria}, the answer is negative in the finitely generated case.

    \begin{corollary}
    Let $f\colon \N \to \mathbb{U}$ be a non-pretentious, finitely generated multiplicative function. Assume there exists a non-pretentious multiplicative function $g\colon \N \to \mathbb{U}$ with $d_{\mathbb{P}}(\mathcal{P}) = 0$, where $\mathcal{P}=\{p\in \mathbb{P}: g(p) \neq 1\}$ such that $f$ pretends to be $g$. Then, if $\mathcal{P}' =\{p\in \mathbb{P}: f(p)\neq 1\}$, we have $\underline{d_{\mathbb{P}}}(\mathcal{P}')=0$. In particular, if the relative density of $\mathcal{P}'$ in the primes exists, it is $0$.
\end{corollary}

Since usually the behaviour of infinitely generated multiplicative functions seems to be completely different from the finitely generated ones, one can expect the previous results to fail in this case, leading us to the
following.
\begin{proposition}\label{prop_example_infinite_generated}
    There exists an infinitely generated non-pretentious multiplicative function $f\colon \N\to \mathbb{U}$ such that if we let $\mathcal{P}=\{p\in \P : f(p) \neq 1\}$, then $\underline{d_{\P}}(\mathcal{P}) >0$ and, for any $\epsilon := \epsilon(x)$ taking values in $(1/\log\log x, 1/2)$, we have $\mathbb{D}(f,1;x^{\epsilon},x) = o_{x\to \infty}(1)$.
\end{proposition}
\begin{proof}
    For any primes $p\in \P$, let $\theta_p = 1/\sqrt{\log\log p}$. We define $f \colon \N \to \mathbb{U}$ as follows. For every prime $p\in \P$, we let $f(p) = e(\theta_p)$, and multiplicatively extend $f$ by letting $f(n) = e\left(\sum_{p^k \Vert n} \theta_p\right)$ for all $n\in \N$. Observe that $f(p) \neq 1$ for all $p\in \P$, hence $\underline{d_{\P}}(\mathcal{P})=1$. Now, for any large enough $x>0$, we have
    \begin{align*}
        \mathbb{D}(f,1;x)^2 &= \sum_{p\leq x} \frac{1-cos(2\pi\theta_p)}{p} = \sum_{p\leq x} \frac{2\pi^2\theta_p^2}{p}+O\left(\frac{\theta_p^4}{p}\right)
        = \sum_{p\leq x}\frac{2\pi^2}{p\log\log p}+O\left(\frac{1}{\log\log x}\right).
    \end{align*}
    Using summation by parts and the PNT, for $x$ large enough, we get 
    \begin{equation*}
        \sum_{p\leq x}\frac{1}{p\log\log p} = \log\log\log x +O\left(\frac{1}{\log\log x}\right).
    \end{equation*}
    It immediately follows that $f$ is non-pretentious since $\mathbb{D}(f,1;x) \to \infty $ as $x\to \infty$. Observe that one can show that $\mathbb{D}(f,\chi(n)n^{it},x) \to \infty$ for non-trivial twisted Dirichlet characters using a similar argument as in the proof of Proposition \ref{prop_example_rotated_liouville} together with the fact that, for $p$ large enough, one has $f(p) \sim 1$. Moreover, the previous calculations imply that whenever $\epsilon \in (1/\log\log x, 1/2)$, we have
    \begin{align*}
        \mathbb{D}(f,1;x^{\epsilon},x)^2 &= -2\pi^2\log\left(1+\frac{\log\epsilon}{\log\log x}\right) +O\left(\frac{1}{\log\log x}\right)\sim \frac{2\pi^2\log(1/\epsilon)}{\log\log x} = o_{x\to \infty}(1),
    \end{align*}
    which concludes the proof.
\end{proof}
Using a similar idea, we give an example of an infinitely generated, non-pretentious multiplicative function for which Corollary \ref{corollary_pretend_elliott} does not hold.
\begin{proposition}\label{prop_example_infinitely_generated_bis}
    There exists an infinitely generated, non-pretentious multiplicative function $f\colon\N\to \mathbb{U}$ such that it pretends to be some non-pretentious multiplicative function $g\colon \N \to \mathbb{U}$ with $d_{\mathbb{P}}(\mathcal{P}) = 0$, where $\mathcal{P}=\{p\in \mathbb{P}: g(p) \neq 1\}$ but, if we let $\mathcal{P}' =\{p\in \mathbb{P}: f(p)\neq 1\}$, we have $\underline{d_{\mathbb{P}}}(\mathcal{P}')>0$.
\end{proposition}
\begin{proof}
    Fix $\mathcal{P}\subset \P$ to be a subset of primes such that $d_{\P}(\mathcal{P})=0$ and $\sum_{p\in \mathcal{P}} \frac{1}{p}=\infty$. Let $g\colon \N \to \mathbb{U}$ be the Liouville-like function $\lambda_{\mathcal{P}}$. Then, $\mathcal{P}=\{p\in \P : g(p) \neq 1\}$ has relative density $0$. Moreover, we know (see \cite[Thm~2.1]{klurman2023elliottsconjectureapplications}) that $g$ is non-pretentious.
    
    Now, choose positive real numbers $\theta_p$ that converge to zero as $p\to \infty$ satisfying
    \begin{equation*}
        \sum_{p\in \P}\frac{\theta_p^2}{p}<\infty,
    \end{equation*}
    but $\theta_p \neq 0$ for all $p\in \P$. For instance, one can take $\theta_p = 1/\log p$. Define $f\colon\N\to \mathbb{U}$ completely multiplicatively by $f(p) = g(p)e(\theta_p)$. We see that $f(p) \neq 1$ for all $p\in \P$. Therefore, if we define $\mathcal{P}' = \{p\in \P : f(p) \neq 1\}$, we clearly have $\underline{d_{\P}}(\mathcal{P}') = 1$. Yet $f$ pretends to be $g$. Indeed,
    \begin{align*}
        \mathbb{D}(f,g;x)^2 &=\sum_{p\leq x} \frac{1-cos(2\pi\theta_p)}{p}\asymp \sum_{p\leq x} \frac{2\pi^2\theta_p^2}{p}.
    \end{align*}
    Thus, we conclude by the choice of $\theta_p$ that $\mathbb{D}(f,g,\infty)<\infty$. Finally, $f$ is also non-pretentious, this follows immediately from the triangle inequality. Indeed, if $f$ pretended to be some twisted Dirichlet character $\chi(n)n^{it}$, then
    \begin{equation*}
        \mathbb{D}(g,\chi(n)n^{it}) \leq \mathbb{D}(g,f)+\mathbb{D}(f,\chi(n)n^{it}),
    \end{equation*}
    which would remain bounded, contradicting the non-pretentiousness of $g$.
\end{proof}

\section{A counterexample to the converse of Theorem \ref{thm_cool_criteria}}\label{sec_a_counter_example_to_the_converse_of}

The goal of this section is to show that the converse of Theorem $\ref{thm_cool_criteria}$ does not hold in general. The underlying idea is that the condition $\underline{d_{\P}}(\mathcal{P})=0$ alone does not provide sufficient control over the distribution of $\mathcal{P}$ within the primes.
We first give, in Theorem \ref{thm_big_goal_example}, an explicit example of a set $\mathcal{P}\subset \P$ with logarithmic relative density $d^{\log}_{\P}(\mathcal{P})=0$ for which the associated Liouville-like function likewise fails to locally pretend to be a twisted Dirichlet character. Thus, even the stronger sparsity condition $d^{\log}_{\P}(\mathcal{P})=0 $, motivated by \eqrefn{eq_densities}, is insufficient to guarantee the desired local pretentiousness.

The following construction provides the basis for our subsequent example.

\begin{lemma}\label{lemma_example_subset_primes}
There exists a set $\mathcal{P}\subseteq \mathbb{P}$ such that
\[
\textnormal{(i)}\ d_{\mathbb{P}}^{\log}(\mathcal{P})=0,
\qquad
\textnormal{(ii)}\ \overline{d_{\mathbb{P}}}(\mathcal{P})>0,
\qquad
\textnormal{(iii)}\ \sum_{p\in\mathcal{P}}\frac{1}{p}=\infty.
\]
\end{lemma}
\begin{proof}
    Let $(a_n)_{n\in \N}$ the sequence defined by $a_n = n^{n^{n^3}}$, $n\in \N$. Let $\mathcal{P}\subset  \P$ be the subset of primes defined by $\mathcal{P} = \{p\in \mathbb{P} : \exists n\in \N$ s.t $a_n^{1/\log\log a_n}<p\leq a_n\}$. We claim that $\mathcal{P}$ satisfies all three properties of the lemma.

    We start by showing that $\overline{d_{\mathbb{P}}}(\mathcal{P}) > 0$. We have for any $n\in \N$
    \begin{equation*}
        \frac{\pi_{\mathcal{P}}(a_n)}{\pi(a_n)}\geq 1-\frac{\pi(a_n^{\frac{1}{\log\log a_n}})}{\pi(a_n)}.
    \end{equation*}
    Now, we see via the PNT that for $n$ large enough 
    \begin{align*}
        \frac{\pi(a_n^{\frac{1}{\log\log a_n}})}{\pi(a_n)} &\sim a_n^{\frac{1}{\log\log a_n}-1}\log\log a_n = \left(n^{n^{n^3}}\right)^{\frac{1}{n^3\log n + \log\log n}-1}(n^3\log n +\log\log n) \xrightarrow{n\to\infty} 0.
    \end{align*}
    Hence, it follows that 
    \begin{align*}
        \overline{d_{\mathbb{P}}}(\mathcal{P}) &= \limsup_{x\to \infty} \frac{\pi_{\mathcal{P}}(x)}{\pi(x)}= \lim_{n\to \infty }\frac{\pi_{\mathcal{P}}(a_n)}{\pi(a_n)} =1.
    \end{align*}

    We now argue that $\sum_{p\in \mathcal{P}} \frac{1}{p}=\infty$. Notice that 
    \begin{equation*}
        \sum_{p\in \mathcal{P}}\frac{1}{p} = \sum_{n\in \N}\sum_{a_n^{\frac{1}{\log\log a_n}}<p \leq a_n}\frac{1}{p},
    \end{equation*}
    and using Merten's theorem (see e.g. \cite[Thm.~I.1.10]{tenenbaum1995introduction}), we have for large enough $n$
    \begin{align*}
        \sum_{a_n^{\frac{1}{\log\log a_n}}<p \leq a_n}\frac{1}{p}&= \log\log a_n - \log\left(\frac{\log a_n}{\log\log a_n}\right) +O\left(\frac{1}{\log a_n^{\frac{1}{\log\log a_n}}}\right)\\
        &= \log\log\log a_n +O\left(\frac{1}{\log a_n^{\frac{1}{\log\log a_n}}}\right) \xrightarrow{n\to \infty} \infty,  
    \end{align*}
    implying that $\sum_{p\in\mathcal{P}}\frac{1}{p} = \infty$.

    Finally, we show $d^{\log}_{\mathbb{P}}(\mathcal{P}) = 0$. We first observe that
    \begin{align*}
        d^{\log}_{\mathbb{P}}(\mathcal{P}) &\leq \lim_{N\to \infty} \frac{1}{\log\log a_N}\sum_{\substack{p\leq a_N\\ p\in \mathcal{P}}}\frac{1}{p}.
    \end{align*}
    Moreover, for any $N\in \N$, one has
    \begin{align*}
         \frac{1}{\log\log a_N}\sum_{\substack{p\leq a_N\\ p\in \mathcal{P}}}\frac{1}{p} &=  \frac{1}{\log\log a_N}\sum_{n\leq N}\sum_{a_n^{\frac{1}{\log\log a_n}}<p\leq a_n}\frac{1}{p}.
    \end{align*}
    Using the previous calculation for the inner sum when $N$ is large enough, one has
  
    \begin{align*}
        \frac{1}{\log\log a_N}\sum_{\substack{p\leq a_N\\ p\in \mathcal{P}}}\frac{1}{p} &= \frac{1}{\log\log a_N}\sum_{n\leq N}\left( \log\log\log a_n + O\left(\frac{\log \log a_n}{\log a_n}\right)\right)\ll \frac{N\log\log\log a_N}{\log\log a_N}.
    \end{align*}
    We now observe that
    \begin{align*}
        \frac{N\log\log\log a_N}{\log\log a_N} &= \frac{N\log\log\log N^{N^{N^3}}}{\log\log N^{N^{N^3}}}\ll \frac{N\log(2N^3\log N)}{N^3\log N +\log\log N}\xrightarrow{N\to \infty} 0,
    \end{align*}
    concluding the proof of the fact that $d^{\log}_{\mathbb{P}}(\mathcal{P}) = 0$.   
\end{proof}
    Observe from \eqrefn{eq_densities} that the condition $d^{\log}_{\P}(\mathcal{P})=0$ implies $\underline{d_{\P}}(\mathcal{P})=0$. Moreover, picking any strictly increasing sequence of integers $(a_n)_{n\in \N}$ such that $a_n \gg n^{n^{n^3}}$ would give rise to a set $\mathcal{P}=\{p\in \P : \exists n\in\N$ s.t $a_n^{1/\log\log a_n}<p\leq a_n\}$ satisfying the properties stated in Lemma \ref{lemma_example_subset_primes}.

We now argue that for a subset of this form, the associated Liouville-like function cannot locally be pretentious.
\begin{theorem}\label{thm_big_goal_example}
    Let $(a_n)_{n\in \N}$ be any strictly increasing sequence of integers such that $a_n \gg n^{n^{n^3}}$, let $\mathcal{P} = \{p\in \P: \exists n\in \N$ s.t. $a_n^{1/\log\log a_n} < p\leq a_n\}$, and let $f\colon \N \to \mathbb{U}$ be defined by $f(n) = \lambda_{\mathcal{P}}(n)$. Then, $f$ is a non-pretentious multiplicative function such that we have, for any Dirichlet character $\chi$, real number $t$, and $\epsilon := \epsilon(N)$ taking values in $(1/(\log\log a_N), 1/2)$ that $\mathbb{D}(f,\chi(n)n^{it};a_N^{\epsilon},a_N) \gg \log 1/\epsilon$ as $N\to \infty$.
\end{theorem}
\begin{proof}
     Let $N\in \N$ be a sufficiently large integer. We start by showing the claim when we consider the constant $1$ function. For any $\epsilon \in (1/\log\log a_N, 1/2)$, by Merten's theorem,
    \begin{align*}
        \mathbb{D}(f,1;a_N^{\epsilon},a_N)^2 &= \sum_{a_N^{\epsilon}<p\leq a_N}\frac{2}{p}= 2\log1/\epsilon + O\left(\frac{1}{\epsilon \log a_N}\right),
    \end{align*}
    and the claim follows from this. If we now consider an arbitrary Dirichlet character $\chi$ of modulus $q$, we clearly have using the PNT for arithmetic progressions
    \begin{align*}
        \mathbb{D}(f,\chi;a_N^{\epsilon},a_N)^2 &\geq \sum_{\substack{0< \ell <q\\ (\ell,q)=1}}\left(1+\Re(\overline{\chi(\ell)}\right)\sum_{\substack{a_N^{\epsilon}<p\leq a_N\\p\equiv \ell \mod q}}\frac{1}{p}\\
        &= \frac{\log 1/\epsilon}{\varphi(q)}\sum_{\substack{0< \ell <q\\ (\ell,q)=1}}\left(1+\Re(\overline{\chi(\ell)}\right) +O\left(\frac{1}{\epsilon \log a_N}\right).
    \end{align*}
    Noticing that $\sum_{0< \ell <q, (\ell,q)=1}\left(1+\Re(\overline{\chi(\ell)}\right)\neq 0$ yields the conclusion. Assuming now that $t\neq 0$ and $\chi$ is again an arbitrary Dirichlet character of modulus $q$, we see that
    \begin{align*}
        \mathbb{D}(f,\chi(n)n^{it},a_N^{\epsilon},a_N)^2 &\geq \sum_{\substack{0< \ell <q\\(\ell,q)=1}}\sum_{\substack{a_N^{\epsilon}<p\leq a_N\\p\equiv \ell\mod q}}\frac{1}{p}+\Re\left(\sum_{\substack{0< \ell <q\\(\ell,q)=1}}\overline{\chi(\ell)}\sum_{\substack{a_N^{\epsilon}<p\leq a_N\\p\equiv \ell\mod q}}\frac{1}{p^{1+it}}\right).
    \end{align*}
    Using again the PNT for arithmetic progressions, we obtain
    \begin{equation*}
        \sum_{\substack{0< \ell <q\\(\ell,q)=1}}\sum_{\substack{a_N^{\epsilon}<p\leq a_N\\p\equiv \ell\mod q}}\frac{1}{p} = \log 1/\epsilon +o_{N\to \infty}(1).
    \end{equation*}
     Moreover, from Lemma \ref{lemma_sum_pit}, we deduce that
    \begin{equation*}
        \sum_{\substack{a_N^{\epsilon}<p\leq a_N\\p\equiv \ell\mod q}}\frac{1}{p^{1+it}} = o_{N\to \infty}(1),
    \end{equation*}
    from which we get $\mathbb{D}(f,\chi(n)n^{it},a_N^{\epsilon},a_N)\gg \log 1/\epsilon$.    
\end{proof}
Combining Theorem \ref{thm_big_goal_example} together with Lemma \ref{lemma_example_subset_primes} gives a constructive counterexample to the converse of Theorem \ref{thm_cool_criteria}.

Finally, we show that for any arbitrary subset of primes $\mathcal{P}$ of upper relative density $1$, its associated Liouville-like function cannot locally be pretentious either, thus generalising Theorem \ref{thm_big_goal_example}.
\begin{theorem}\label{thm_bis_big_goal_densities}
     Let $f\colon \N \to \mathbb{U}$ be defined by $f(n)=\lambda_{\mathcal{P}}(n)$, where $\mathcal{P}\subset \mathbb{P}$ has upper relative density equal to $1$. Then, there exists a strictly increasing sequence $(a_N)_{N\in \N} \subseteq\N$ of integers such that, for any Dirichlet character $\chi$, real number $t$ and $\epsilon:= \epsilon(N)$ taking values in $ (1/(\log\log a_N),1/2)$ we have $\mathbb{D}(f,\chi(n)n^{it};a_N^{\epsilon},a_N) \gg \log 1/\epsilon$ as $N\to \infty$.
\end{theorem}

This result follows from the following lemma.
\begin{lemma}\label{lemma_bis_subset_primes_contain_arbitrarily_long_intervals}
    Let $\mathcal{P}\subset \mathbb{P}$ such that $\overline{d_{\mathbb{P}}}(\mathcal{P}) = 1$. Then, \begin{equation*}
        d^{I}_{\mathbb{P}}(\mathcal{P}) := \limsup_{N\to\infty} \frac{|\mathcal{P}\cap [N^{\frac{1}{\log\log N}},N]|}{|\mathbb{P}\cap[N^{\frac{1}{\log\log N}},N]|} =1.
    \end{equation*}
\end{lemma}
\begin{proof}
Let $\mathcal{P}\subset \mathbb{P}$, and assume $\overline{d_{\mathbb{P}}}(\mathcal{P}) = 1$. We want to show that
\begin{equation*}
    d^{I}_{\mathbb{P}}(\mathcal{P}) = \limsup_{N\to \infty}\frac{\pi_{\mathcal{P}}(N) - \pi_{\mathcal{P}}(N^{1/\log\log N})}{\pi(N) - \pi(N^{1/\log\log N})} = 1.
\end{equation*}
Notice that we have
\begin{align*}
    d^{I}_{\mathbb{P}}(\mathcal{P}) &\geq \limsup_{N\to \infty}\frac{\pi_{\mathcal{P}}(N) - \pi_{\mathcal{P}}(N^{1/\log\log N})}{\pi(N) }\\&\geq \limsup_{N\to \infty} \frac{\pi_{\mathcal{P}}(N)}{\pi(N)}-\limsup_{N\to \infty} \frac{\pi_{\mathcal{P}}(N^{1/\log\log N})}{\pi(N)}
    = 1-\limsup_{N\to \infty} \frac{\pi_{\mathcal{P}}(N^{1/\log\log N})}{\pi(N)}.
\end{align*}
Using the PNT, we see that, as $N\to \infty$, we have
\begin{align*}
    \frac{\pi_{\mathcal{P}}(N^{1/\log\log N})}{\pi(N)} &\ll\frac{N^{1/\log\log N}\log N}{N\log\left(N^{1/\log\log N}\right)}= N^{(1/\log\log N) -1}\log\log N,
\end{align*}
and we conclude by noticing that $N^{(1/\log\log N)-1 }\log\log N = o_{N\to \infty}(1)$.
\end{proof}

\begin{proof}[Proof of Theorem $\ref{thm_bis_big_goal_densities}$]
    Let $\mathcal{P}\subset \mathbb{P}$ be a set such that $\overline{d_{\mathbb{P}}}(\mathcal{P})=1$. We know from Lemma \ref{lemma_bis_subset_primes_contain_arbitrarily_long_intervals} that $d^I_{\mathbb{P}}(\mathcal{P})=1$. Let $(u_N)_{N\in \N}$ be the sequence defined by
    \begin{align*}
        u_N = \frac{|\mathcal{P}\cap [N^{\frac{1}{\log\log N}},N]|}{|\mathbb{P}\cap [N^{\frac{1}{\log\log N}},N]|},
    \end{align*}
    so that $d^I_{\mathbb{P}}(\mathcal{P})=\limsup_{N\to \infty} u_N = 1$. Consider $(a_N)_{N\in \N} \subset \N$ to be a strictly increasing sequence of integers such that $u_{a_N} > 1-1/N$ for all $N\in \N$. Let $\chi$ be a Dirichlet character and fix $t\in \R$. For any $N\in \N$, fix $\epsilon\in (1/\log\log a_N, 1/2)$. We will show that $\mathbb{D}(f,\chi(n)n^{it};a_N^{\epsilon},a_N) \gg \log1/\epsilon$ as $N\to \infty$. We have
    \begin{equation}\label{eq_2}
        \mathbb{D}(f,\chi(n)n^{it};a_N^{\epsilon},a_N)^2
        = \sum_{\substack{a_N^{\epsilon}<p\leq a_N \\ p\in \mathcal{P}}}\frac{1+\Re(\chi(p)p^{it})}{p}+\sum_{\substack{a_N^{\epsilon}<p\leq a_N \\ p\notin  \mathcal{P}}}\frac{1-\Re(\chi(p)p^{it})}{p}.
    \end{equation}
    We now show that the second sum on the right-hand side of \eqrefn{eq_2} is $ o_{N\to \infty}(1)$. In fact, we have 
    \begin{align*}
        \left| \sum_{\substack{a_N^{\epsilon}<p\leq a_N \\ p\notin  \mathcal{P}}}\frac{1-\Re(\chi(p)p^{it})}{p}\right| 
        &\leq \frac{2}{a_N^{\epsilon}} \sum_{\substack{a_N^{\epsilon}<p\leq a_N \\ p\notin  \mathcal{P}}} 1 \\&=\frac{2}{a_N^{\epsilon}} \left( |\mathbb{P}\cap [a_N^{\epsilon},a_N]| - \frac{|\mathcal{P}\cap [a_N^{\epsilon},a_N]|}{|\mathbb{P}\cap [a_N^{\epsilon},a_N]|}|\mathbb{P}\cap [a_N^{\epsilon},a_N]|\right).
    \end{align*}
    Observe that
    \begin{align*}
        1\geq  \frac{|\mathcal{P}\cap [a_N^{\epsilon},a_N]|}{|\mathbb{P}\cap [a_N^{\epsilon},a_N]|}&= \frac{|\mathcal{P}\cap [a_N^{\epsilon},a_N]|+|\mathbb{P}\cap [a_N^{\frac{1}{\log\log a_N}},a_N^{\epsilon}]|-|\mathbb{P}\cap [a_N^{\frac{1}{\log\log a_N}},a_N^{\epsilon}]|}{|\mathbb{P}\cap [a_N^{\epsilon},a_N]|}\\
        &\geq\frac{|\mathcal{P}\cap [a_N^{\frac{1}{\log\log a_N}},a_N]|-|\mathbb{P}\cap [a_N^{\frac{1}{\log\log a_N}},a_N^{\epsilon}]|}{|\mathbb{P}\cap [a_N^{\epsilon},a_N]|}.
    \end{align*}
    Using now $d^I_{\mathbb{P}}(\mathcal{P})=1$, we have $|\mathcal{P}\cap [a_N^{\frac{1}{\log\log a_N}},a_N]| \sim |\mathbb{P}\cap [a_N^{\frac{1}{\log\log a_N}},a_N]|$ as $N\to \infty$, hence
    \begin{align*}
        \frac{|\mathcal{P}\cap [a_N^{\frac{1}{\log\log a_N}},a_N]|-|\mathbb{P}\cap [a_N^{\frac{1}{\log\log a_N}},a_N^{\epsilon}]|}{|\mathbb{P}\cap [a_N^{\epsilon},a_N]|}&\sim \frac{|\mathbb{P}\cap [a_N^{\frac{1}{\log\log a_N}},a_N]|-|\mathbb{P}\cap [a_N^{\frac{1}{\log\log a_N}},a_N^{\epsilon}]|}{|\mathbb{P}\cap [a_N^{\epsilon},a_N]|} = 1.
    \end{align*}
    In other words, we have $ |\mathcal{P}\cap [a_N^{\epsilon},a_N]| \xrightarrow{N\to \infty}  |\mathbb{P}\cap [a_N^{\epsilon},a_N]|$, hence the conclusion follows.
    We now consider the first sum on the right-hand side of \eqrefn{eq_2}. Notice that
    \begin{align*}
        \sum_{\substack{a_N^{\epsilon}<p\leq a_N \\ p\in \mathcal{P}}}\frac{1+\Re(\chi(p)p^{it})}{p}&= \sum_{a_N^{\epsilon}<p\leq a_N }\frac{1+\Re(\chi(p)p^{it})}{p}\,-\sum_{\substack{a_N^{\epsilon}<p\leq a_N \\ p\notin \mathcal{P}}}\frac{1+\Re(\chi(p)p^{it})}{p}.
    \end{align*}
    Using Lemma \ref{lemma_technical_result}, we see that
    \begin{align*}
        \left| \sum_{\substack{a_N^{\epsilon}<p\leq a_N \\ p\notin \mathcal{P}}}\frac{1+\Re(\chi(p)p^{it})}{p}\right| &\leq \sum_{\substack{a_N^{\epsilon}<p\leq a_N \\ p\notin \mathcal{P}}}\frac{2}{p} = o_{N\to\infty}(1).
    \end{align*}
     Moreover, similarly as in the proof of Theorem \ref{thm_big_goal_example}, we can show using Lemma \ref{lemma_sum_pit} that for any $t\in \R$ we have $\sum_{a_N^{\epsilon}<p\leq a_N }\frac{1+\Re(\chi(p)p^{it})}{p} \gg \log 1/\epsilon$, giving the claim.
\end{proof}

A natural direction for future work would be to generalise Theorem \ref{thm_bis_big_goal_densities} to the setting where $\mathcal{P}\subset \P$ has positive upper relative density. 

\section{Acknowledgments}
The author is grateful to Michael Baake, Michael Coons and Florian Richter for helpful discussions and comments on the manuscript. He thanks the TRR 358 for support during two visits in 2026, where this manuscript was completed.


\begin{thebibliography}{99}

\bibitem{Chowla1966TheRH}
S.~Chowla,
The Riemann hypothesis and Hilbert's tenth problem,
\textit{Amer. Math. Monthly} \textbf{73} (1966), 906--907.

\bibitem{Elliott1994OnTC}
P.~D.~T.~A. Elliott,
On the correlation of multiplicative functions and the sum of additive functions,
\textit{Mem. Amer. Math. Soc.} \textbf{112} (1994), no.~538.

\bibitem{granvillebook}
A.~Granville and K.~Soundararajan,
\textit{Multiplicative Number Theory: The Pretentious Approach},
In preparation.

\bibitem{Halsz1968berDM}
G.~Halász,
Über die Mittelwerte multiplikativer zahlentheoretischer Funktionen,
\textit{Acta Math. Acad. Sci. Hungar.} \textbf{19} (1968), 365--403.

\bibitem{klurman2023elliottsconjectureapplications}
O.~Klurman, A.~P.~Mangerel, and J.~Teräväinen,
On Elliott's conjecture and applications,
Preprint, arXiv:2304.05344 (2023).

\bibitem{Matom_ki_2015}
K.~Matomäki, M.~Radziwiłł, and T.~Tao,
An averaged form of Chowla's conjecture,
\textit{Algebra Number Theory} \textbf{9} (2015), no.~9, 2167--2196. arXiv:1503.05121.


\bibitem{IASsarnak}
P.~Sarnak,
Three lectures on the Möbius function, randomness and dynamics,
Lecture notes, Institute for Advanced Study (2011).

\bibitem{Tao_Tera_2019}
T.~Tao and J.~Teräväinen,
The structure of logarithmically averaged correlations of multiplicative functions,
with applications to the Chowla and Elliott conjectures,
\textit{Duke Math. J.} \textbf{168} (2019), no.~11, 2053--2150. arXiv:1708.02610.


\bibitem{Tao_2019}
T.~Tao and J.~Teräväinen,
The structure of correlations of multiplicative functions at almost all scales,
with applications to the Chowla and Elliott conjectures,
\textit{Algebra Number Theory} \textbf{13} (2019), no.~9, 2103--2150. arXiv:1809.02518.


\bibitem{tao2016logarithmicallyaveragedchowlaelliott}
T.~Tao,
The logarithmically averaged Chowla and Elliott conjectures for two-point correlations,
\textit{Forum Math. Pi} \textbf{4} (2016), e8: 1--36. arXiv:1509.05422.


\bibitem{tenenbaum1995introduction}
G.~Tenenbaum,
\textit{Introduction to Analytic and Probabilistic Number Theory},
Cambridge University Press, Cambridge, (1995).

\end{thebibliography}
\end{document}